\documentclass[11pt]{article}

\usepackage[margin=1in]{geometry}
\usepackage{amsmath,amssymb,amsfonts,amsthm,mathtools}
\usepackage{enumitem}
\usepackage{microtype}
\usepackage{float}
\usepackage{tikz}
\usepackage{placeins}
\usetikzlibrary{arrows.meta,calc,positioning}
\usepackage[hidelinks]{hyperref}
\hypersetup{
  pdftitle={Information and Locality in Cayley Graphs},
  pdfauthor={Ming-Hsuan Kang and Yu-Hsuan Hsieh},
  pdfsubject={Separating patterns and local windows in Cayley graphs},
  pdfkeywords={Cayley graph, separating pattern, de Bruijn sequence, carbon torus, graph identification}
}

\newcommand{\F}{\mathbb{F}}
\newcommand{\Z}{\mathbb{Z}}
\newcommand{\Cay}{\operatorname{Cay}}
\newcommand{\Span}{\operatorname{span}}
\newcommand{\GL}{\operatorname{GL}}
\newcommand{\sep}{\operatorname{sep}}
\newcommand{\csep}{\operatorname{csep}}
\newcommand{\Dih}{\operatorname{Dih}}
\newcommand{\End}{\operatorname{End}}
\newcommand{\ord}{\operatorname{ord}}

\newtheorem{theorem}{Theorem}[section]
\newtheorem{proposition}[theorem]{Proposition}
\newtheorem{computationalproposition}[theorem]{Computational Proposition}
\newtheorem{lemma}[theorem]{Lemma}

\theoremstyle{definition}
\newtheorem{definition}[theorem]{Definition}

\theoremstyle{remark}
\newtheorem{remark}[theorem]{Remark}
\newtheorem{question}[theorem]{Question}

\title{Information and Locality in Cayley Graphs}
\author{%
  Ming-Hsuan Kang\thanks{Corresponding author. Email: \texttt{kmsming@gmail.com}.}
  \and
  Yu-Hsuan Hsieh\thanks{Email: \texttt{benny01237@gmail.com}.}
  \\[4pt]
  \small Department of Applied Mathematics, National Yang Ming Chiao Tung
  University, Hsinchu, Taiwan
}
\date{}

\begin{document}
\maketitle

\begin{abstract}
A de Bruijn sequence is the cyclic prototype of a Cayley-graph observation
problem: when does the ordered label word on a translated window $gY$
determine the vertex $g$?  We distinguish three parameters.  The unrestricted
number $\sep_q(G)$ minimizes an arbitrary separating pattern; the connected
number $\csep_q(G,S)$ requires a connected Cayley window containing
$Y_S=\{1\}\cup S$; and the one-step number $\chi_1(G,S)$ fixes $Y_S$ and
minimizes the alphabet.  Thus $\sep_q$ is a group-level baseline,
$\csep_q$ measures the cost of locality, and $\chi_1$ tests the smallest
prescribed local window.

The organizing theme is the tension between information and locality.  Carbon
tori test the gap between $\sep_q$ and $\csep_q$: for generalized dihedral
groups $\F_{\ell^d}^{\times}\rtimes C_2$ we prove, for odd prime powers
$\ell$, the sharp baseline $\sep_\ell=d+1$ and construct connected zig-zag
windows, while the order-$14$
Heawood torus satisfies $\sep_4=2$ and $\csep_4=4$.  The spherical $A_5$
example and a finite simple-group comparison test the fixed one-step window:
explicit symmetric cubic generating tuples give $\chi_1(A_5,S)=3$ and
$\chi_1(\operatorname{PSL}_2(\F_7),S)=4$, both at the counting bound, with
structured matrix-coefficient certificates.  Cyclic-coset packings,
finite-field coordinates, and restricted matrix coefficients are used only
as the construction tools these two examples require.
\end{abstract}


\section{Introduction}

A de Bruijn sequence may be viewed as an optimal local coordinate system on a
cycle.  A cyclic word over a $q$-symbol alphabet can be chosen so that every
word of length $k$ occurs exactly once as a consecutive
block~\cite{debruijn,aardenne-debruijn}.
Thus an observer reading the same ordered interval
$(0,1,\ldots,k-1)$ at every position can determine its location on a cycle of
length $q^k$.  More broadly, this suggests a local self-location problem on a
symmetric network: vertices carry symbols from a small alphabet, and position
is inferred from the labels seen in a fixed translated configuration.  We use
this only as combinatorial motivation.  The essential point is that both the
visible symbols and their relative positions carry information.

We formulate this problem on a finite Cayley graph.  Let $G$ be a finite
group, let $[q]$ be an alphabet of $q$ symbols, let $f:G\to[q]$ be a
labelling, and let $Y=(y_1,\ldots,y_k)$ be an ordered pattern of distinct
elements of $G$.  At $g\in G$ the translated observation is
\[
    W_{f,Y}(g)=\bigl(f(gy_1),\ldots,f(gy_k)\bigr).
\]
The pair $(f,Y)$ is separating when these words are distinct for all $g$.
After an ordered generating tuple $S=(s_1,\ldots,s_r)$ has been chosen, the
right Cayley graph $\Cay(G,S)$ supplies the geometry: the generators are named
directions, and one may require the observation window to be connected.

Three parameters separate the resulting questions.  For a fixed alphabet,
$\sep_q(G)$ minimizes the size of an arbitrary separating pattern.  The graph
parameter $\csep_q(G,S)$ requires a connected Cayley window containing the
identity and the generators.  Finally, $\chi_1(G,S)$ fixes the one-step window
$Y_S=\{1\}\cup S$ and minimizes the alphabet size.  These parameters measure,
respectively, an information-theoretic group baseline, the cost of respecting
the Cayley geometry, and the information carried by the smallest prescribed
local neighborhood.  The last quantity remembers the named generator
directions, so it can depend on the chosen Cayley structure rather than only
on the underlying abstract graph.  The problem lies beside established
graph-identification models such as identifying codes and locally
identifying colorings, which use unordered detector or color information
rather than ordered translated words; Remark~\ref{rem:graph-identification}
makes this comparison precise.

The constructions are systematic as well as equivariant.  Cyclic-coset
packings produce path windows, and finite-field coordinates encode the
associated labellings by evaluating a single linear functional along a group
orbit.  For a fixed one-step window, matrix coefficients
$f(g)=\lambda(\rho(g)v)$ provide the analogous mechanism.  In both cases an
ordered window collects several translates of one scalar observation, and
separation becomes an injectivity statement.  Algebra therefore organizes
the construction of the local code; it is not the object of the problem.

Our main geometric setting is the family of carbon tori, periodic cubic
honeycomb quotients with generalized-dihedral Cayley models
$G_N=(\Z^2/N)\rtimes C_2$.  Products of their involutory generator directions
give translations, so cyclic path constructions lift to connected zig-zag
windows in the cubic graph.  Finite-field coordinates also give sharp
unrestricted baselines.  The smallest nondegenerate quotient already exhibits
the cost of locality: for the order-$14$ Heawood carbon torus,
\[
    \sep_4(G_H)=2,
    \qquad
    \csep_4(G_H,S_H)=4.
\]

The complementary one-step problem is tested on the truncated icosahedral
Cayley graph on $A_5$, a spherical realization of Buckminsterfullerene studied
in~\cite{kang-lin}, and on a nonplanar finite simple comparison from
$\operatorname{PSL}_2(\F_7)$.  Exact certificates give one-step alphabet
numbers $3$ and $4$, respectively, both attaining the counting bound.  These
examples test the information carried by a prescribed local neighborhood,
whereas the carbon tori test the geometry of connected windows.

The emphasis is on the distinction among unrestricted, connected, and
one-step windows, together with construction principles that can be
transferred to other Cayley graphs.  Sections~\ref{sec:cyclic}
and~\ref{sec:toolbox} develop the construction tools,
Sections~\ref{sec:spherical} and~\ref{sec:carbon} treat the two graph families,
and Section~\ref{sec:questions} records the remaining questions.
Table~\ref{tab:exact-values} summarizes the exact values proved below.

\begin{table}[H]
\centering
\small
\renewcommand{\arraystretch}{1.15}
\begin{tabular}{c|c|c}
Graph or group & Invariant & Exact value \\
\hline
$A_5$, $S=(a,b,b^{-1})$ & $\chi_1(G,S)$ & $3$ \\
$\operatorname{PSL}_2(\F_7)$, $S=(a,b,b^{-1})$ & $\chi_1(G,S)$ & $4$ \\
$C_7\rtimes C_2$ (Heawood torus) & $\sep_4(G)$ & $2$ \\
$C_7\rtimes C_2$ (Heawood torus) & $\csep_4(G,S)$ & $4$ \\
\hline
$\F_{q^n}^{\times}$, $q^n>2$ (cyclic benchmark) & $\sep_q(G)$ & $n$ \\
$\F_{\ell^d}^{\times}\rtimes C_2$, $\ell$ odd (group baseline)
  & $\sep_\ell(G)$ & $d+1$
\end{tabular}
\caption{Exact values established in the main examples, followed by supporting
and comparison benchmarks.}
\label{tab:exact-values}
\end{table}

\section{Separating windows in Cayley graphs}

Throughout, the groups on which our parameters are evaluated and their
quotient graphs are finite.  The identity of a group is denoted by $1$.  If
$S=(s_1,\ldots,s_r)$ is an ordered generating tuple of distinct nonidentity
elements, then $\Cay(G,S)$ denotes the right Cayley graph: its vertices are
the elements of $G$, and the $s_i$-edge from $g$ leads to $gs_i$.  We regard
this graph as undirected when the underlying generating set is inverse-closed;
in set operations, a generating tuple means its underlying set.  We write
$C_n$ for the cyclic group of order $n$, $D_n$ for the dihedral group of order
$2n$, $A_n$ for the alternating group, and $K_n$ for the complete graph.  The
symbol $\square$ denotes the Cartesian product of graphs, and $\rtimes$
denotes a semidirect product.  For an integer $q\geq2$, write
$[q]=\{0,1,\ldots,q-1\}$ for a fixed alphabet of size $q$.  When finite-field
constructions are invoked, $q$ is a prime power, $\F_q$ is the field of $q$
elements, and $\F_q^{\times}=\F_q\setminus\{0\}$ is its multiplicative group.
We freely identify any $q$-element alphabet, such as $\F_q$, with $[q]$.

Fix a Cayley graph $X=\Cay(G,S)$ and an integer $q\geq2$.
All words are read in a fixed order on their window coordinates.

\begin{definition}[Translated word and separating pattern]
Let $f:G\to [q]$ be a labelling and let
\[
    Y=(y_1,\ldots,y_k)
\]
be an ordered tuple of distinct elements of $G$, called a \emph{pattern}.
Write $|Y|=k$, and use $Y$ also for its underlying set when discussing
containment or induced subgraphs.  The $Y$-word of $f$ at $g\in G$ is
\[
    W_{f,Y}(g)=\bigl(f(gy_1),\ldots,f(gy_k)\bigr)\in[q]^k.
\]
The pair $(f,Y)$ is called \emph{separating} if the map
\[
    g\mapsto W_{f,Y}(g)
\]
is injective.
\end{definition}

\begin{definition}[Separating pattern number]
The $q$-ary separating pattern number of $G$ is
\[
    \sep_q(G)=\min\{|Y|:\text{there exists }f:G\to[q]\text{ such that }(f,Y)\text{ is separating}\}.
\]
\end{definition}

This unrestricted number deliberately forgets the edges of the Cayley graph.
It will serve as an auxiliary baseline for the graph parameters below.

\begin{definition}[Cayley window]
Let $S$ be a chosen generating set or generating tuple.  A pattern
$Y\subseteq G$ is called a \emph{Cayley window} if
\[
    1\in Y,
    \qquad
    S\subseteq Y,
\]
and the subgraph of $\Cay(G,S)$ induced by $Y$ is connected after edge
orientations, if any, are forgotten.  The connected separating window number is
\[
\begin{split}
    \csep_q(G,S)=\min\{ |Y|:\;&Y\text{ is a Cayley window, and}
    \\ &\exists f:G\to[q]\text{ such that }(f,Y)\text{ is separating}\}.
\end{split}
\]
Thus for a non-symmetric tuple $S$, connectedness refers to the underlying
undirected graph of the right Cayley digraph.
\end{definition}

Thus
\[
    \sep_q(G)\leq \csep_q(G,S).
\]
The difference measures the cost of forcing the observation pattern to be local in the chosen Cayley graph.

\begin{definition}[One-step separating alphabet number]
Let $S=(s_1,\ldots,s_r)$ be an ordered generating tuple of $G$, and put
\[
    Y_S=\{1\}\cup S=\{1,s_1,\ldots,s_r\}.
\]
The displayed order, beginning with the identity, is used for the word
coordinates.
The one-step separating alphabet number is the least integer $q\geq2$ such
that there exists a labelling $f:G\to[q]$ for which
\[
    g\mapsto \bigl(f(g),f(gs_1),\ldots,f(gs_r)\bigr)
\]
is injective.  We denote this number by $\chi_1(G,S)$.
\end{definition}

The invariant $\chi_1(G,S)$ fixes the smallest natural local window and asks how many colors are needed.  This is closer to graph-identification colorings.  By contrast, $\sep_q(G)$ fixes the alphabet size and minimizes the pattern size.

\begin{proposition}[Basic bounds]
\label{prop:basic-bounds}
For every finite group $G$,
\[
    \sep_q(G)\geq \lceil \log_q |G|\rceil.
\]
For a Cayley graph $\Cay(G,S)$,
\[
    \csep_q(G,S)
    \geq
    \max\{\lceil\log_q|G|\rceil, |S\cup\{1\}|\}.
\]
If $S=(s_1,\ldots,s_r)$, then
\[
    \chi_1(G,S)\geq \left\lceil |G|^{1/(r+1)}\right\rceil.
\]
Moreover, separating patterns always exist; in fact $\sep_q(G)\leq |G|$ for every $q\geq 2$.
\end{proposition}

\begin{proof}
The lower bounds follow because a word of length $k$ over $[q]$ has only
$q^k$ possible values.  For existence, take $Y=G$ and let $f$ be the delta
function at the identity:
\[
    f(x)=\begin{cases}1,&x=1,\\ 0,&x\neq 1.\end{cases}
\]
Then $W_{f,G}(g)$ has its unique nonzero coordinate at $y=g^{-1}$, so it determines $g$.
\end{proof}

\begin{remark}[Relation with graph-identification parameters]
\label{rem:graph-identification}
Several established graph-identification models are close in motivation but
use different observations.  An identifying code distinguishes all vertices
by the intersections of a chosen code with their closed
neighborhoods~\cite{karpovsky}; size bounds under additional graph
restrictions, including triangle-free graphs, form a substantial part of that
literature~\cite{foucaud}.  A locating-dominating set instead uses open
neighborhood intersections to distinguish vertices outside the chosen
set~\cite{slater}.  A locating coloring distinguishes vertices
by their distance vectors to color classes~\cite{chartrand}; a
neighbor-locating coloring compares the sets of colors appearing in open
neighborhoods~\cite{alcon}; and a locally identifying coloring is a proper
coloring that distinguishes adjacent vertices with distinct closed
neighborhoods by the sets of colors appearing there~\cite{esperet}.

The present observation rule retains different information.  A fixed ordered
pattern $Y$ is translated equivariantly by the Cayley action, and the observer
reads the ordered word $(f(gy))_{y\in Y}$ rather than a set of colors or
detectors.  This leads to two complementary optimizations: fix the alphabet
and minimize an arbitrary pattern or a connected Cayley window
($\sep_q$ and $\csep_q$), or fix the one-step window $Y_S$ and minimize the
alphabet ($\chi_1$).  Thus the parameters here are not reformulations of the
locating-coloring parameters, even though all belong to the broader graph
identification setting.
\end{remark}

\section{Cyclic constructions for carbon tori}
\label{sec:cyclic}

We first develop a graph construction suggested by the classical De Bruijn
sequence viewpoint.  Choose a long cyclic direction in the Cayley graph,
split the vertices into its coset cycles, and encode all of those cycles
simultaneously.  The resulting observation window is a path; in a honeycomb
graph it becomes a two-edge zig-zag after subdividing each translation step.

\subsection{Cyclic-coset path windows}

Let $s\in G$ have order $m$, and put
\[
    H=\langle s\rangle.
\]
The right cosets of $H$ partition $G$.  On each coset $gH$, the sequence
\[
    g,
    gs,
    gs^2,
    \ldots,
    gs^{m-1}
\]
is a cycle in the $s$-direction.

\begin{definition}[Coset cycle packing]
For integers $m\geq1$ and $1\leq k\leq m$, let $M_q(m,k)$ denote the maximum
number of cyclic
$q$-ary words of length $m$ whose cyclic length-$k$ factors are all distinct,
both within each word and across different words.
\end{definition}

Our convention is that the $q$-ary De Bruijn digraph of order $k$ has
length-$(k-1)$ words as vertices and length-$k$ words as directed edges: the
edge $x_1\cdots x_k$ runs from $x_1\cdots x_{k-1}$ to
$x_2\cdots x_k$.  A cyclic word traces a directed closed walk through its
successive length-$k$ factors.  The required distinctness makes this walk an
edge-simple closed directed trail.  Hence $M_q(m,k)$ is equivalently the
maximum number of pairwise edge-disjoint closed directed trails of length
$m$ in this digraph.

\begin{proposition}[Cyclic-coset bound]
Let $2\leq k\leq m$, let $s\in G$ have order $m$, let
$H=\langle s\rangle$, and let
\[
    r=[G:H].
\]
Thus $r=|G|/|H|$ is the number of right cosets of $H$ in $G$.
If
\[
    r\leq M_q(m,k),
\]
then
\[
    Y_k=(1,s,s^2,\ldots,s^{k-1})
\]
is a separating pattern for some labelling $f:G\to[q]$.  In particular,
\[
    \sep_q(G)\leq k.
\]
If $s\in S$ for a Cayley graph $\Cay(G,S)$, then
\[
    Y_k\cup S
\]
is a Cayley window, and hence
\[
    \csep_q(G,S)\leq k+|S|-1.
\]
\end{proposition}

\begin{proof}
Choose representatives $g_1,\ldots,g_r$ for the right cosets of $H$.  By the hypothesis, choose $r$ cyclic words $C_1,\ldots,C_r$ of length $m$ such that all cyclic length-$k$ factors are distinct across all chosen words.  Write
\[
    C_i=(c_{i,0},c_{i,1},\ldots,c_{i,m-1}).
\]
Define
\[
    f(g_i s^j)=c_{i,j},
\]
with indices taken modulo $m$.  Then the $Y_k$-word at $g_i s^j$ is exactly the cyclic length-$k$ factor of $C_i$ beginning at $j$.  These factors are all distinct by construction, so $Y_k$ separates $G$.

If $s\in S$, then $Y_k$ is a path in the $s$-direction and contains $1$ and $s$.  Adding the remaining elements of $S$ keeps the induced subgraph connected, since all elements of $S$ are adjacent to $1$.
\end{proof}

The preceding proposition is useful but conditional.  The following finite-field version gives a more explicit supply of coset cycles.

\begin{theorem}[Finite-field cyclic-coset construction]
Let $q$ be a prime power, let $1\leq k\leq m$, let $s\in G$ have order $m$, let
$H=\langle s\rangle$, and put
$r=[G:H]$.  Suppose there exists an element
$\theta\in \F_{q^k}^{\times}$ whose multiplicative order
$\ord(\theta)$ satisfies
\[
    \ord(\theta)=m
\]
and
\[
    \{1,\theta,\ldots,\theta^{k-1}\}
\]
is an $\F_q$-basis of $\F_{q^k}$.  If
\[
    r\leq \frac{q^k-1}{m},
\]
then
\[
    \sep_q(G)\leq k.
\]
More precisely, $Y_k=(1,s,\ldots,s^{k-1})$ is separating for a suitable labelling.
\end{theorem}

\begin{proof}
Let $K=\F_{q^k}$ and let $L=\langle\theta\rangle\leq K^\times$.  Since $|L|=m$, the number of cosets of $L$ in $K^\times$ is $(q^k-1)/m$.  Choose distinct coset representatives
\[
    c_1L,\ldots,c_rL.
\]
Let $\psi:K\to\F_q$ be a nonzero $\F_q$-linear functional.  For coset representatives $g_1,
\ldots,g_r$ of $G/H$, define
\[
    f(g_i s^j)=\psi(c_i\theta^j).
\]
The word at $g_i s^j$ is
\[
    \bigl(\psi(c_i\theta^j),\psi(c_i\theta^{j+1}),\ldots,\psi(c_i\theta^{j+k-1})\bigr).
\]
Because $\{1,\theta,\ldots,
\theta^{k-1}\}$ is a basis of $K$ over $\F_q$, the map
\[
    x\mapsto (\psi(x),\psi(x\theta),\ldots,\psi(x\theta^{k-1}))
\]
is injective on $K$.  Thus equality of two such words implies
\[
    c_i\theta^j=c_{i'}\theta^{j'},
\]
which forces $i=i'$ and $j=j'$ because the cosets $c_iL$ are distinct.  Hence the words separate $G$.
\end{proof}

\subsection{Finite-field cyclic benchmark}
\label{sec:models}

Finite cyclic groups provide the basic graph model for separating windows.
Finite-field coordinates give particularly transparent labellings; the
algebra here is only a device for constructing the graph labelling.

\begin{lemma}[Basis observation]
Let $q$ be a prime power, let $n\geq1$, let $K=\F_{q^n}$, and let
$\psi:K\to\F_q$ be a nonzero
$\F_q$-linear functional.  If
\[
    Y=(y_1,\ldots,y_n)
\]
is an $\F_q$-basis of $K$, then
\[
    x\mapsto (\psi(xy_1),\ldots,\psi(xy_n))
\]
is an $\F_q$-linear isomorphism from $K$ to $\F_q^n$.
\end{lemma}

\begin{proof}
If the displayed map sends $x$ to zero and $x\neq 0$, then $xY=(xy_1,\ldots,xy_n)$ is again a basis of $K$, so $\psi$ vanishes on a basis, hence $\psi=0$, a contradiction.  Thus the map is injective, and the two vector spaces have the same dimension.
\end{proof}

\begin{proposition}[Cyclic finite fields]
For every prime power $q$ and $n\geq 1$,
\[
    \sep_q(\F_{q^n}^{\times})=n,
\]
except for the trivial case $q=2,n=1$, where the group has one element.
\end{proposition}

\begin{proof}
The lower bound follows from
\[
    |\F_{q^n}^{\times}|=q^n-1>q^{n-1}
\]
in the nontrivial case.  For the upper bound choose a primitive element
$\alpha\in \F_{q^n}^{\times}$, meaning a generator of this multiplicative
group, and use
\[
    Y=(1,\alpha,\ldots,\alpha^{n-1}).
\]
The basis observation shows that this pattern separates $\F_{q^n}^{\times}$.
\end{proof}

This is the usual maximal-length finite-field analogue of a De Bruijn
sequence.  It is exhaustive in the sense that the image consists of all
nonzero vectors of $\F_q^n$, and it is exactly the instance of the basis
observation used for the finite-field carbon-torus benchmark of
Section~\ref{sec:carbon}.

\section{One-step matrix coefficients for the spherical example}
\label{sec:toolbox}

This section records the one algebraic device needed for the spherical
example.  The Cayley graph and its ordered one-step window remain fixed;
matrix coefficients supply a certificate that the translated words are
distinct.

Let $q$ be a prime power and let $V$ be a finite-dimensional
$\F_q$-vector space.  We write
$V^*=\operatorname{Hom}_{\F_q}(V,\F_q)$ for its dual space,
$\End(V)$ for its endomorphism algebra, and $\GL(V)$ for its group of
invertible linear maps; when $V=\F_q^d$, we also write $\GL_d(\F_q)$.
Let $\rho:G\to\GL(V)$ be a representation, meaning a group homomorphism.  For
$v\in V$ and $\lambda\in V^*$, define the
matrix-coefficient labelling
\[
    f_{\lambda,v}(g)=\lambda(\rho(g)v).
\]
Here and below $\Span_{\F_q}$ denotes $\F_q$-linear span.  For the prescribed
local window, only the following restricted orbit matters.

Let $S=(s_1,\ldots,s_r)$ be an ordered generating tuple and let
\[
    Y_S=\{1\}\cup S=\{1,s_1,\ldots,s_r\}.
\]
For a representation $\rho:G\to\GL(V)$ and $v\in V$, set
\[
    U_S=\Span_{\F_q}\{v,\rho(s_1)v,\ldots,\rho(s_r)v\}.
\]

\begin{proposition}[One-step matrix-coefficient criterion]
\label{prop:one-step-mc}
For the labelling
\[
    f(g)=\lambda(\rho(g)v),
\]
the one-step window $Y_S$ is separating if and only if the functionals
\[
    \lambda\circ\rho(g)|_{U_S},
    \qquad g\in G,
\]
    are pairwise distinct.  In particular, if this pairwise-distinctness
    requirement (the \emph{restricted regular-orbit condition}) holds, then
\[
    \chi_1(G,S)\leq q.
\]
If additionally
\[
    (q-1)^{r+1}<|G|\leq q^{r+1},
\]
then
\[
    \chi_1(G,S)=q.
\]
\end{proposition}

\begin{proof}
The one-step word at $g$ records the values of
$\lambda\circ\rho(g)|_{U_S}$ on the spanning tuple
\[
    v,\rho(s_1)v,\ldots,\rho(s_r)v.
\]
Thus two one-step words agree if and only if the corresponding restricted
functionals agree.  This proves the equivalence and the upper bound.  The
equality statement follows from the counting lower bound.
\end{proof}

\begin{lemma}[Dimension screen for matrix-coefficient certificates]
\label{lem:mc-dimension-screen}
Let $k\geq1$, let $Y=(y_1,\ldots,y_k)$, and put
\[
  U_Y=\Span_{\F_q}\{\rho(y_1)v,\ldots,\rho(y_k)v\}.
\]
If the matrix coefficient $f(g)=\lambda(\rho(g)v)$ separates $G$ on $Y$,
then
\[
  q^{\dim U_Y}\geq |G|.
\]
In particular, a representation of dimension $d$ cannot support such a
certificate when $q^d<|G|$.
\end{lemma}

\begin{proof}
The word at $g$ is determined by the restricted functional
$\lambda\circ\rho(g)|_{U_Y}$.  There are only
$q^{\dim U_Y}$ linear functionals on $U_Y$, whereas separation requires at
least $|G|$ distinct restrictions.
\end{proof}

\section{A spherical Cayley graph and a finite simple benchmark}
\label{sec:spherical}

Our spherical model is the truncated icosahedral graph, the graph of
Buckminsterfullerene.  Here $A_5$ is the alternating group on
$\{0,1,2,3,4\}$.  It has the Cayley realization
\[
    X_{\mathrm{sph}}=
    \Cay\bigl(A_5,\{a,b,b^{-1}\}\bigr),
    \qquad
    a=(0\ 1)(2\ 3),\quad b=(0\ 1\ 2\ 3\ 4),
\]
used in the spherical drawing theory of~\cite{kang-lin}.  The $b$-cycles are
the pentagonal directions, while the relation $(ab)^3=1$ gives alternating
six-step cycles.  Thus the ordered window $(1,a,b,b^{-1})$ is precisely a
vertex together with its three incident Cayley directions.

For comparison,
$\operatorname{PSL}_2(\F_7)=\operatorname{SL}_2(\F_7)/\{\pm I\}$ is the
projective special linear group, where $\operatorname{SL}_2(\F_7)$ is the
group of determinant-one $2\times2$ matrices and $I$ is the identity matrix.
Also,
$\mathbb{P}^1(\F_7)=\F_7\cup\{\infty\}$ is its projective line.  The graph on
$\operatorname{PSL}_2(\F_7)$ below is not asserted to be planar;
it is included as a finite simple-group comparison.  This separates the local
window phenomenon from the special spherical embedding of the $A_5$ graph.

For the $A_5$ tuple above, $a$, $b$, and $ab$ have orders $2$, $5$, and
$3$, respectively.  For the comparison graph, take
$\operatorname{PSL}_2(\F_7)$ acting on $\mathbb{P}^1(\F_7)$ with the
projective transformations
\[
    a(x)=-\frac1x,
    \qquad
    b(x)=x+1.
\]
Here $a$, $b$, and $ab$ have orders $2$, $7$, and $3$.  In each case the
displayed elements generate the stated group and $S=(a,b,b^{-1})$ is a
symmetric cubic generating tuple.

\begin{computationalproposition}[Finite simple-group benchmarks]
\label{prop:simple-group-benchmarks}
For the preceding generating tuples,
\[
    \chi_1(A_5,S)=3,
    \qquad
    \chi_1(\operatorname{PSL}_2(\F_7),S)=4.
\]
\end{computationalproposition}

\begin{proof}
The counting bound gives the lower estimates because
$2^4<60=|A_5|$ and $3^4<168=|\operatorname{PSL}_2(\F_7)|$.

For $A_5$, let $\rho_6:A_5\to\GL_6(\F_3)$ be the permutation representation
on the cosets of the normalizer of a Sylow $5$-subgroup, equivalently the
action of $A_5\cong\operatorname{PSL}_2(\F_5)$ on
$\mathbb{P}^1(\F_5)$.  The supplementary certificate gives
$v\in\F_3^6$ and $\lambda\in(\F_3^6)^*$ satisfying the restricted
regular-orbit condition of Proposition~\ref{prop:one-step-mc}; hence
$\chi_1(A_5,S)\leq3$.

For $\operatorname{PSL}_2(\F_7)$, let
$\rho_8:\operatorname{PSL}_2(\F_7)\to\GL_8(\F_4)$ be the permutation
representation on $\mathbb{P}^1(\F_7)$.  A second supplementary certificate
provides $v\in\F_4^8$ and $\lambda\in(\F_4^8)^*$ satisfying the same
condition, so $\chi_1(\operatorname{PSL}_2(\F_7),S)\leq4$.  The exact
certificate check also confirms that the corresponding one-step span has
dimension $4$.  Independent explicit labelling certificates provide a
second verification of both values.
\end{proof}

\begin{remark}[Certificate versus module-selection principle]
The dimension screen explains why a $3$-dimensional
$\operatorname{PSL}_2(\F_7)$ module cannot work over $\F_4$, since
$4^3<168$.  It does not explain why the degree-six action of $A_5$ contains a
successful pair $(v,\lambda)$ while the natural degree-five permutation
module does not, or why the degree-eight projective action is effective for
$\operatorname{PSL}_2(\F_7)$.  The certificates prove these facts, not a
general module-selection principle.
\end{remark}

\section{Honeycomb toroidal Cayley graphs and carbon tori}
\label{sec:carbon}

Carbon tori are the main family for the connected-window problem.  The
Cayley-graph viewpoint on fullerenes is developed in~\cite{kang}; Alspach and Dean show
that honeycomb toroidal graphs admit generalized-dihedral Cayley
descriptions~\cite{alspach}.  From the present graph-theoretic viewpoint, a
carbon torus is a cubic Cayley graph embedded on a torus with hexagonal faces.
Its local window has four vertices,
but its long translation directions wrap around the torus.  The central
questions are therefore to compare the group-level baseline $\sep_q(G_N)$
with the connected number $\csep_q(G_N,S_N)$ and to determine the one-step
alphabet number $\chi_1(G_N,S_N)$.  A useful Cayley model arises from
generalized dihedral groups.

\subsection{The generalized dihedral model}

Let $A$ be a finite abelian group, written additively in this subsection.  Put
\[
    \Dih(A)=A\rtimes C_2,
\]
where $C_2=\langle\tau\rangle$ acts on $A$ by $x\mapsto-x$.  In coordinates
we identify $\tau^\epsilon$ with $\epsilon\in\{0,1\}$, and all second
coordinates are read modulo $2$.  Thus
\[
    (x,\epsilon)(y,\delta)=(x+(-1)^\epsilon y,\epsilon+\delta).
\]

In the honeycomb-torus model, one takes
\[
    A_N=\Z^2/N
\]
for a finite-index sublattice $N\subseteq\Z^2$, and
\[
    G_N=A_N\rtimes C_2.
\]
A natural cubic generating set is
\[
    S_N=
    \{(v_1,1),(v_1+v_2,1),(v_2,1)\},
\]
where $v_1,v_2$ are the images of the standard basis vectors of $\Z^2$.
These generators are involutions.  We call such a honeycomb quotient
\emph{nondegenerate} when the three displayed directions remain distinct and
no hexagonal face collapses; in the finite examples below, the latter
condition is checked by the absence of $4$-cycles.  A nondegenerate quotient
is a finite quotient of the hexagonal lattice.

The ordered one-step window at $g$ is
$gY_S=(g,gs_1,gs_2,gs_3)$.  Figure~\ref{fig:honeycomb-torus} shows how a
honeycomb patch is closed to form a toroidal quotient.

\begin{figure}[H]
\centering
\begin{tikzpicture}[scale=0.72, every node/.style={font=\small}]
  \newcommand{\Apt}[2]{({1.732*(#1)+0.866*(#2)},{1.5*(#2)})}
  \newcommand{\Bpt}[2]{({1.732*(#1)+0.866*(#2)-1.732},{1.5*(#2)-1})}
  \tikzset{
    Avertex/.style={circle,fill=black,inner sep=1.2pt},
    Bvertex/.style={circle,fill=white,draw=black,line width=0.45pt,inner sep=1.1pt}
  }

  \begin{scope}
    \clip (-1.25,-1.15) rectangle (10.15,5.25);

    \foreach \i in {-3,-2,-1,0,1,2,3,4,5,6,7} {
      \foreach \j in {-3,-2,-1,0,1,2,3,4,5,6} {
        \draw[gray!60,line width=0.65pt] \Apt{\i}{\j}--\Bpt{\i+1}{\j};
        \draw[gray!60,line width=0.65pt] \Apt{\i}{\j}--\Bpt{\i+1}{\j+1};
        \draw[gray!60,line width=0.65pt] \Apt{\i}{\j}--\Bpt{\i}{\j+1};
      }
    }
    \foreach \i in {-3,-2,-1,0,1,2,3,4,5,6,7} {
      \foreach \j in {-3,-2,-1,0,1,2,3,4,5,6} {
        \node[Avertex] at \Apt{\i}{\j} {};
        \node[Bvertex] at \Bpt{\i}{\j} {};
      }
    }

    \coordinate (P00) at (-0.45,-0.50);
    \coordinate (P10) at (6.478,-0.50);
    \coordinate (P01) at (2.148,4.00);
    \coordinate (P11) at (9.076,4.00);
    \draw[blue!75,line width=1.15pt] (P00)--(P10)--(P11)--(P01)--cycle;

    \draw[-{Latex[length=2.6mm]},blue!75,line width=0.95pt]
      ($(P00)!0.10!(P10)$)--($(P00)!0.30!(P10)$)
      node[midway,below=3pt] {$v_1$};
    \draw[-{Latex[length=2.6mm]},blue!75,line width=0.95pt]
      ($(P00)!0.10!(P01)$)--($(P00)!0.30!(P01)$)
      node[midway,left=3pt] {$v_2$};

    \draw[-{Latex[length=3mm]},red!75,line width=1.0pt]
      ($(P00)!0.42!(P10)+(0,-0.28)$)--($(P00)!0.62!(P10)+(0,-0.28)$);
    \draw[-{Latex[length=3mm]},red!75,line width=1.0pt]
      ($(P01)!0.42!(P11)+(0,0.28)$)--($(P01)!0.62!(P11)+(0,0.28)$);
    \draw[-{Latex[length=3mm]},orange!85!black,line width=1.0pt]
      ($(P00)!0.42!(P01)+(-0.25,0)$)--($(P00)!0.62!(P01)+(-0.25,0)$);
    \draw[-{Latex[length=3mm]},orange!85!black,line width=1.0pt]
      ($(P10)!0.42!(P11)+(0.25,0)$)--($(P10)!0.62!(P11)+(0.25,0)$);
  \end{scope}
\end{tikzpicture}
\caption{A clipped honeycomb tiling with a fundamental parallelogram.  Black
and white vertices represent the two bipartite layers $(u,0)$ and
$(u,1)$, with $u\in A_N$.  Each pair of matching side arrows is identified;
the quotient
is a finite honeycomb toroidal Cayley graph, or carbon torus.}
\label{fig:honeycomb-torus}
\end{figure}
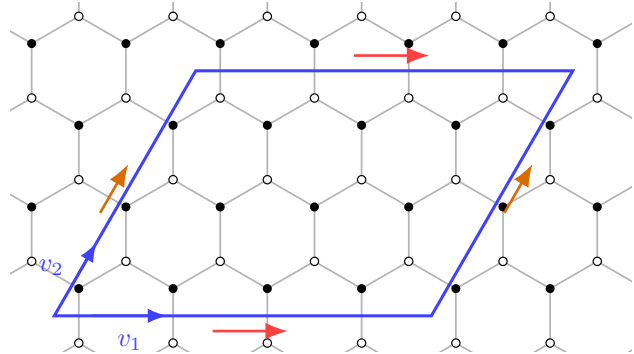

\FloatBarrier
\subsection{Orientation doubling}

We return temporarily to multiplicative notation for a finite abelian group $A$.  Let
\[
    \Dih(A)=A\rtimes C_2
\]
where $C_2=\langle \tau\rangle$ acts by inversion.

\begin{lemma}[Orientation doubling]
Let $q$ be an odd prime power.  Suppose $h:A\to\F_q$ and
$Y=(y_1,\ldots,y_k)\subseteq A$ satisfy:
\begin{enumerate}[label=\textup{(\arabic*)}]
    \item $1\in Y$;
    \item $Y$ separates $A$ under $h$;
    \item $Y^{-1}$ also separates $A$ under $h$.
\end{enumerate}
Define
\[
    f(a)=h(a),
    \qquad
    f(a\tau)=h(a)+1.
\]
Then
\[
    Y'=Y\cup\{\tau\}
\]
separates $\Dih(A)$.  Hence
\[
    \sep_q(\Dih(A))\leq |Y|+1.
\]
\end{lemma}

\begin{proof}
For $a\in A$,
\[
    W(a)=\bigl((h(ay))_{y\in Y},h(a)+1\bigr).
\]
For $a\tau\in A\tau$, using $\tau y=y^{-1}\tau$,
\[
    W(a\tau)=\bigl((h(ay^{-1})+1)_{y\in Y},h(a)\bigr).
\]
Assumptions~(2) and~(3) separate rotations and reflections within their own layers.  If a rotation word equaled a reflection word, the coordinates indexed by $1$ and by $\tau$ would give
\[
    h(a)=h(b)+1,
    \qquad
    h(a)+1=h(b),
\]
which implies $2=0$, impossible since $q$ is odd.
\end{proof}

\begin{theorem}[Finite-field generalized dihedral groups]
Let $\ell$ be an odd prime power and let $d\geq1$.  Then
\[
    \sep_\ell(\F_{\ell^d}^{\times}\rtimes C_2)=d+1,
\]
where $C_2$ acts by inversion.
\end{theorem}

\begin{proof}
Let $K=\F_{\ell^d}$ and choose a primitive element $\alpha\in K^\times$.  Let $\psi:K\to\F_\ell$ be nonzero and set $h(x)=\psi(x)$.  The pattern
\[
    Y=(1,\alpha,\ldots,\alpha^{d-1})
\]
separates $K^\times$.  Since $\alpha^{-1}$ also has degree $d$ over
$\F_\ell$, the inverse powers
$1,\alpha^{-1},\ldots,\alpha^{-(d-1)}$ form a basis, so the inverse pattern
also separates.  Orientation doubling gives the upper bound $d+1$.  For the
lower bound,
\[
    |K^\times\rtimes C_2|=2(\ell^d-1),
\]
and for odd $\ell$,
\[
    \ell^d<2(\ell^d-1)<\ell^{d+1}.
\]
Thus the counting lower bound is $d+1$.
\end{proof}

This theorem is a group-level benchmark, not yet a connected-window bound;
the next subsection realizes it concretely and addresses the connectivity
question.

\subsection{Realization as carbon tori}

Let $\ell$ be an odd prime power, let $K=\F_{\ell^d}$, and choose a primitive
element $\alpha\in K^\times$.  Choose an integer $r$ and define
\[
    \phi:\Z^2\to K^\times,
    \qquad
    \phi(v_1)=\alpha,
    \quad
    \phi(v_2)=\alpha^r.
\]
Then $\phi$ is surjective, and with $N=\ker\phi$,
\[
    \Z^2/N\cong K^\times.
\]
The associated honeycomb Cayley group is
\[
    G_N=(\Z^2/N)\rtimes C_2\cong K^\times\rtimes C_2.
\]
The natural generating set becomes
\[
    S=\{\alpha\tau,\alpha^{r+1}\tau,\alpha^r\tau\}.
\]
When these three involutions are distinct and the quotient has no collapsed hexagonal faces, $\Cay(G_N,S)$ is a honeycomb toroidal Cayley graph.  The previous theorem gives
\[
    \sep_\ell(G_N)=d+1.
\]

The optimal pattern
\[
    (1,\alpha,\ldots,\alpha^{d-1},\tau)
\]
computes the group-level invariant, but it need not be connected in the honeycomb graph.  This is precisely why $\csep_q(G,S)$ is a separate graph-geometric refinement.

\subsection{Zig-zag cyclic windows in carbon tori}

In the honeycomb model the generators $s_1,s_2,s_3$ are involutions, so $\langle s_i\rangle$ has order $2$.  However products such as
\[
    t=s_1s_2
\]
lie in the translation subgroup $A_N$ and may have large order.  If $t$ has order $m$, the cyclic-coset construction applied to $H=\langle t\rangle$ gives a separating path window in the $t$-direction.  In the original cubic graph, one step by $t=s_1s_2$ is a two-edge zig-zag.

The intermediate vertices $t^js_1$ visited along this path give the following
precise graph bound.

\begin{proposition}[Zig-zag Cayley-window bound]
In a carbon-torus group $G_N$, let $S=(s_1,s_2,s_3)$ be the cubic honeycomb
generating tuple, put $t=s_1s_2$, let $m=\ord(t)$, let $2\leq k\leq m$, and
suppose
\[
    Y_k=\{1,t,\ldots,t^{k-1}\}
\]
is separating for some $q$-ary labelling.  Then
\[
    \widetilde Y_k=
    Y_k\cup\{t^js_1:0\leq j\leq k-2\}\cup S
\]
is a connected separating Cayley window.  In particular,
\[
    \csep_q(G_N,S)\leq 2k+1.
\]
\end{proposition}

\begin{proof}
For $0\leq j\leq k-2$, the two Cayley edges
\[
    t^j\;\mathord{-}\;t^js_1
    \;\mathord{-}\;t^js_1s_2=t^{j+1}
\]
join consecutive translation vertices.  These vertices therefore form a
zig-zag path containing $Y_k$.  Adding $S$ preserves connectedness because
each generator is adjacent to $1$.  The enlarged window remains separating
because it contains the separating subpattern $Y_k$.  It has at
most $k$ translation vertices, $k-1$ intermediate vertices, and the two
generators other than the already present $s_1$, for a total of at most
$2k+1$.
\end{proof}

Combined with the cyclic-coset criteria of Section~\ref{sec:cyclic}, this proposition gives
explicit connected windows whenever the translation direction supports a
suitable cycle packing.  It is the direct graph-theoretic analogue of
encoding several parallel cycles by De Bruijn words.

\subsection{The smallest honeycomb torus: the Heawood quotient}

The zig-zag bound gives connected windows but does not measure their excess
over the unrestricted baseline.  The comparison can already be made exactly
on the smallest nondegenerate quotient of the honeycomb lattice.

\begin{computationalproposition}[The Heawood carbon-torus gap]
\label{prop:smallest-torus-gap}
Let $A=C_7$, written additively, and put
\[
  G_H=\Dih(C_7)=C_7\rtimes C_2,
  \qquad |G_H|=14.
\]
Let $v_1=1$, $v_2=3$ in $C_7$ and take the honeycomb generating triple
\[
  S_H=\{(1,1),(4,1),(3,1)\}.
\]
Then $\Cay(G_H,S_H)$ is the Heawood graph: it is connected, bipartite, and
trivalent of girth $6$.  It is the smallest nondegenerate honeycomb quotient
arising from a finite-index sublattice of $\Z^2$.  Moreover,
\[
  \sep_4(G_H)=2,
  \qquad
  \csep_4(G_H,S_H)=4.
\]
In particular $\csep_4(G_H,S_H)=2\,\sep_4(G_H)$.
\end{computationalproposition}

\begin{proof}
The quotient map
\[
  \Z^2\longrightarrow C_7,
  \qquad (x,y)\longmapsto x+3y
\]
has a finite-index, non-diagonal kernel and sends the three honeycomb
directions to $D=\{1,4,3\}$.  The six ordered differences $d-d'$ with
$d,d'\in D$ and $d\ne d'$ are precisely the six nonzero elements of $C_7$.
Thus the three generators are distinct and the bipartite Cayley graph has no
$4$-cycle.  It is connected because these differences generate $C_7$, and it
has a $6$-cycle, for example
\[
 (0,0),(1,1),(4,0),(5,1),(1,0),(4,1),(0,0).
\]
Hence its girth is $6$.  Conversely, in any nondegenerate cubic honeycomb
quotient the six ordered differences must be distinct and nonzero; therefore
$|A|\geq7$ and the full generalized-dihedral graph has at least $14$
vertices.  This proves minimality.

The counting and structural bounds give $\sep_4(G_H)\geq2$ and
$\csep_4(G_H,S_H)\geq4$.  Matching labellings for a two-element unrestricted
window and for the one-step window prove the reverse inequalities.  The
accompanying supplement gives and verifies both labellings.
\end{proof}

\begin{remark}
The three shifts $\{1,3,4\}$ form a cyclic $(7,3,1)$ difference set.  The two
bipartite layers are accordingly the points and lines of the projective
plane of order $2$, the Fano plane, which identifies the graph as the
Heawood graph.  The proposition shows that the cyclic coincidence
$\sep_q=\csep_q$ fails already at the first possible honeycomb quotient.
\end{remark}

\section{Graph-theoretic refinements and open problems}
\label{sec:questions}

The main unresolved problems concern the geometry imposed by the Cayley
generators rather than the existence of unrestricted separating patterns.

\begin{question}[Carbon-torus windows]
For honeycomb toroidal Cayley graphs
$G_N=(\Z^2/N)\rtimes C_2$, determine or bound
$\csep_q(G_N,S_N)$ and $\chi_1(G_N,S_N)$.  The Heawood example has
$\csep_4=2\sep_4$; does this ratio remain bounded as the torus grows?  For the
one-step window, when is the counting bound
\[
    \chi_1(G_N,S_N)\geq \left\lceil |G_N|^{1/4}\right\rceil
\]
sharp, and how does the threshold alphabet depend on $N$?
\end{question}

\begin{question}[Spherical windows and matrix coefficients]
For which cubic planar Cayley graphs is the counting lower bound for
$\chi_1(G,S)$ sharp?  Can representation theory predict which modules contain
a separating matrix coefficient for a prescribed one-step window?  In
particular, explain conceptually why the degree-six action used for $A_5$
succeeds and why the degree-eight projective action is effective for
$\operatorname{PSL}_2(\F_7)$.
\end{question}

\begin{question}[Cyclic directions]
For the cyclic group $C_N$, determine $\sep_q(C_N)$ and characterize when the
cyclic-coset finite-field construction is sharp.  Translate such results into
exact connected-window bounds when the cycle occurs as a direction in a
Cayley graph.
\end{question}

\section*{Computational certificates}
The exact certificates and reproduction instructions for the computational
propositions are available in the
\href{https://github.com/kmsming-prog/information-and-locality-in-cayley-graphs}{\underline{companion repository}};
see \nolinkurl{README.md}.

\section*{Acknowledgements}
The work of Ming-Hsuan Kang was supported by the National Science and
Technology Council of Taiwan under grant NSTC 115-2115-M-A49-010.

\subsection*{Use of generative AI and AI-assisted technologies in
the manuscript preparation process}
During the preparation of this work, the authors used OpenAI ChatGPT
and Codex for literature organization, language editing, and
software-assisted verification.  After using these tools, the authors
reviewed and edited the content as needed and take full responsibility
for the content of the publication.

\end{document}